\documentclass[11pt]{article}
\usepackage{amsmath}
\usepackage{amsthm}
\usepackage{amssymb,amsfonts}
\usepackage{hyperref}
\usepackage{latexsym}
\usepackage{todonotes}
\usepackage{nicefrac}
\usepackage{epsfig}
\usepackage{stmaryrd}
\usepackage{setspace}
\usepackage{tikz-cd}
\usepackage{enumerate}
\usepackage[all]{xypic}
\usepackage{bbm,ifpdf,tikz,mathtools}
\ifpdf
\usepackage{pdfsync}
\fi
\usetikzlibrary{positioning,matrix,arrows,decorations.pathmorphing}
\usepackage{enumitem}

\usetikzlibrary{positioning,arrows}
\usetikzlibrary{decorations.markings}
\usepackage{caption, subcaption}
\usetikzlibrary{arrows.meta,
                bbox}

\newtheorem{theorem}{Theorem}[section]

\newtheorem{lemma}[theorem]{Lemma}
\newtheorem{proposition}[theorem]{Proposition}
\newtheorem{corollary}[theorem]{Corollary}
\newtheorem{claim}[theorem]{Claim}

\newtheorem*{main-result}{Main-Result}
{
\theoremstyle{definition}

\newtheorem{example}[theorem]{Example}

\newtheorem{remark}[theorem]{Remark}

}

\newcommand{\becircled}{\mathaccent "7017}
\newcommand{\excise}[1]{}

\newcommand{\Coker}{\operatorname{Coker}}
\newcommand{\Ker}{\operatorname{Ker}}

\newcommand{\Sep}{\operatorname{Sep}}

\renewcommand{\and}{\qquad\text{and}\qquad}

\newcommand{\bigmid}{\;\Big{|}\;}

\newcommand{\Z}{\mathbb{Z}}
\newcommand{\Q}{\mathbb{Q}}

\newcommand{\R}{\mathbb{R}}
\newcommand{\C}{\mathbb{C}}

\renewcommand{\H}{\operatorname{H}}

\newcommand{\cK}{\mathcal{K}}

\newcommand{\cG}{\mathcal{G}}
\newcommand{\cC}{\mathcal{C}}
\newcommand{\cT}{\mathcal{T}}
\newcommand{\cI}{\mathcal{I}}
\newcommand{\cJ}{\mathcal{J}}
\newcommand{\cN}{\mathcal{N}}
\newcommand{\cM}{\mathcal{M}}

\renewcommand{\a}{\alpha}

\newcommand{\cA}{\mathcal{A}}

\renewcommand{\tau}{\mathcal{T}}

\newcommand{\galen}{\todo[inline,color=violet!20]}

\newcommand{\cone}{{\operatorname{cone}}}

\renewcommand{\cN}{\mathcal{N}}

\renewcommand{\cL}{\mathcal{L}}

\newcommand{\Rees}{\operatorname{Rees}}
\newcommand{\VG}{\operatorname{VG}}
\newcommand{\GR}{\operatorname{GR}}
\newcommand{\gr}{\operatorname{gr}}

\newcommand{\Tor}{\operatorname{Tor}}

\newcommand{\init}{{\sf in}}

\begin{document}
\spacing{1.2}
\noindent{\Large\bf Equivariant cohomology and conditional oriented matroids}\\

\noindent{\bf Galen Dorpalen-Barry}\\
Faculty of Mathematics, Ruhr-Universit\"at Bochum, D-44801 Bochum, Germany
\vspace{.1in}

\noindent{\bf Nicholas Proudfoot}\footnote{Supported by NSF grants DMS-1954050, DMS-2039316, and DMS-2053243.}\\
Department of Mathematics, University of Oregon, Eugene, OR 97403
\vspace{.1in}

\noindent{\bf Jidong Wang}\\
Department of Mathematics, University of Texas, Austin, TX 78712\\

{\small
\begin{quote}
\noindent {\em Abstract.}
We give a cohomological interpretation of the Heaviside filtration on the Varchenko--Gelfand ring of a pair $(\cA,\cK)$,
where $\cA$ is a real hyperplane arrangement and $\cK$ is a convex open subset of the ambient vector space.
This builds on work of the first author,
who studied the filtration from a purely algebraic perspective, as well as work of Moseley, who gave a cohomological interpretation
in the special case where $\cK$ is the ambient vector space.  We also define the Gelfand--Rybnikov ring of a conditional oriented matroid, which simultaneously generalizes the Gelfand--Rybnikov ring of an oriented matroid and the aforementioned
Varchenko--Gelfand ring of a pair.  We give purely combinatorial presentations of the ring, its associated graded, and its Rees algebra.
\end{quote} }

\section{Introduction}
In the first half of this paper, the basic object of study is a pair consisting of a hyperplane arrangement in a real vector
space, and a convex open set in that vector space.
We study the Varchenko--Gelfand ring of such a pair, along with its Heaviside filtration, which was introduced by
the first author \cite{db}.  We give
a cohomological interpretation of the Varchenko--Gelfand ring, its associated graded, and its Rees algebra,
generalizing work of de Longueville and Schultz \cite{dLS} and Moseley \cite{moseley} in the case where the convex open set is equal to the vector space itself.

The second half of the paper is devoted to giving combinatorial presentations for these rings.
When the convex set is equal to the vector
space, the rings depend only on the oriented matroid associated with the hyperplane arrangement, and the definitions and
presentations were extended to arbitrary oriented matroids by Gelfand and Rybvnikov \cite{GelfandRybnikov} and Cordovil \cite{Co}.
Introducing the convex open set requires generalizing from oriented matroids to conditional oriented matroids,
introduced by Bandelt, Chepoi, and Knauer \cite{bck2018}.
We define the Gelfand--Rybnikov algebra of a conditional oriented matroid, along with its Heaviside filtration, and we give  presentations for this algebra, its associated graded, and its Rees algebra.

\subsection{Topology}\label{sec:intro top}
Let $V$ be a finite dimensional vector space over $\R$ and $\cA$ a finite set of affine hyperplanes in $V$, and
consider the complement $$M_1(\cA) := V\setminus\bigcup_{H\in \cA} H,$$
which is simply the disjoint union of the chambers.
The {\bf Varchenko--Gelfand ring} $\VG(\cA)$ 
is defined as the ring of locally constant $\Z$-valued functions on $M_1(\cA)$.
This is a boring ring with an interesting filtration:
it is generated as a ring by {\bf Heaviside functions}, which take the value 1 on one side of a given hyperplane and 0 on the other side,
and we define $F_k(\cA)\subset \VG(\cA)$ be the subgroup generated by polynomial expressions in the Heaviside functions of degree at most $k$.
Varchenko and Gelfand \cite{VG} computed the relations between the Heaviside functions.

For any ring $R$ equipped with an increasing filtration $F_0\subset F_1\subset\cdots\subset R$, one can define the {\bf associated graded}
$$\gr R:= \bigoplus_{k\geq 0} F_k/F_{k-1}$$
and the {\bf Rees algebra}
$$\Rees R := \bigoplus_{k\geq 0} u^kF_k\subset R\otimes\Z[u].$$
The Rees algebra is a torsion-free graded module over the polynomial ring $\Z[u]$, and we have canonical isomorphisms
$$\Rees R/\langle u-1\rangle\cong R\and \Rees R/\langle u\rangle\cong \gr R\,.\,\footnote{We will always take the degree of $u$
to be 2, which means that the isomorphism $\Rees R/\langle u\rangle\cong \gr R$ halves degrees.}$$
The geometric meaning of the Heaviside filtration of $\VG(\cA)$, along with its associated graded and its Rees algebra, was explained in a paper of Moseley \cite{moseley}.
For each $H\in\cA$, let $H_0$ be the linear hyperplane obtained by translating $H$ to the origin.  Let
$$H\otimes\R^3 := \{(x,y,z)\in V\otimes\R^3\mid x\in H\;\text{and}\; y,z\in H_0\},$$ and consider the space
$$M_3(\cA) := V\otimes\R^3\setminus\bigcup_{H\in \cA} H\otimes\R^3.$$
This space admits an action of $T := U(1)$ by identifying $\R^3$ with $\R\times\C$ and letting $T$ act on $\C$ by scalar multiplication;
the fixed point set of this action can be identified with the space $M_1(\cA)$.
Moseley showed that we have isomorphisms
\begin{align*}
\VG(\cA)\otimes\Q &\cong \H^*(M_3(\cA)^T; \Q)\\
\gr \VG(\cA)\otimes\Q &\cong \H^*(M_3(\cA); \Q)\\
\Rees \VG(\cA)\otimes\Q &\cong \H^*_T(M_3(\cA); \Q),
\end{align*}
the latter being an isomorphism of graded algebras over $H^*(BT; \Q) \cong \Q[u]$.

The first of the three isomorphisms above is immediate from the definition of $\VG(\cA)$.
When all of the hyperplanes pass through the origin, the second isomorphism
can be obtained by comparing the results of Varchenko and
Gelfand with the presentation of $\H^*(M_3(\cA); \Q)$ due to de Longueville and Schultz \cite[Corollary 5.6]{dLS}.  The most interesting is the last isomorphism, which interpolates between the first two (see Section \ref{sec:eq}).

Our goal is to generalize these results to a larger class of rings and spaces, and also to work with coefficients in $\Z$ rather than in $\Q$.
Fix an open, convex subset $\cK\subset V$, and consider the spaces
$$M_1(\cA,\cK) := M_1(\cA)\cap \cK\and M_3(\cA,\cK) := \big\{(x,y,z)\in M_3(\cA)\mid x\in\cK\big\}.$$
Note that we still have an action of $T$ on $M_3(\cA,\cK)$ with fixed point set isomorphic to $M_1(\cA,\cK)$.
We define the {\bf Varchenko--Gelfand ring of the pair} $(\cA,\cK)$  to be the ring $\VG(\cA,\cK)$
of locally constant $\Z$-valued functions on $M_1(\cA,\cK)$; this ring
was introduced and studied by the first author \cite{db}.
Our first main result is the following theorem.

\begin{theorem}\label{cohomology}
We have canonical isomorphisms
\begin{align*}
\VG(\cA,\cK) &\cong H^*(M_3(\cA,\cK)^T; \Z)\\
\gr \VG(\cA,\cK) &\cong \H^*(M_3(\cA,\cK); \Z)\\
\Rees \VG(\cA,\cK) &\cong \H^*_T(M_3(\cA,\cK); \Z),
\end{align*}
the latter being an isomorphism of graded algebras over $H^*(BT; \Z) \cong \Z[u]$.
\end{theorem}

\begin{remark}
If we take $\cK = V$, then $M_1(\cA,\cK) = M_1(\cA)$ and $M_3(\cA,\cK) = M_3(\cA)$.  We then recover Moseley's result
by tensoring with $\Q$.
\end{remark}

\subsection{Combinatorics}\label{sec:intro comb}
Our proof of Theorem \ref{cohomology} is purely topological, and does not require us to give presentations of any of the rings involved.  That said, each of the three rings in Theorem \ref{cohomology} admits a nice combinatorial presentation,
which is the focus of the second half of our paper.

In the case where $\cK = V$ and all hyperplanes pass through the origin,
the presentations depend only on the {\bf oriented matroid} determined by $\cA$.
Indeed, Gelfand and Rybnikov
\cite{GelfandRybnikov} defined a filtered ring associated with any oriented matroid, generalizing the
Varchenko--Gelfand ring with its Heaviside filtration, and gave a presentation generalizing the one in \cite{VG}.
Independently, Cordovil gave a presentation for the associated graded of this filtered ring \cite{Co}.

Just as the combinatorial data of a central real hyperplane arrangement is captured by an oriented matroid,
the combinatorial essence of a pair $(\cA,\cK)$ is captured by a {\bf conditional oriented matroid}, introduced by Bandelt, Chepoi, and Knauer \cite{bck2018}.
We define the Gelfand--Rybnikov ring of a conditional oriented matroid in a way that generalizes
both the Gelfand--Rybnikov ring of an oriented matroid and the Varchenko--Gelfand ring of a pair ($\cA,\cK)$.
In Theorem \ref{thm:presentations}, we give presentations for this ring, its associated graded, and its Rees algebra, extending the work of \cite{GelfandRybnikov, Co} to conditional oriented matroids.

Before stating the theorem, we review some definitions.
Let $\cI$ be a finite set.
A {\bf signed set} is an ordered pair $X = (X^+,X^-)$ of disjoint subsets of $\cI$.
The {\bf support} of a signed set $X = (X^+,X^-)$ is the unsigned set $\underline{X} := X^+ \cup X^-$.
For any $i\in\cI$, we write $X_i = \pm$ if $i\in X^\pm$, and $X_i = 0$ if $i\notin\underline{X}$.
We write $-X$ to denote the opposite signed set $-X = (X^-,X^+)$, so that $(-X)_i = -X_i$.
The {\bf separating set} of a pair of signed sets $X,Y$ is the set of coordinates in the intersection of the supports
at which $X$ and $Y$ differ:
\[
\Sep(X,Y) := \{i\in \cI\mid X_i = - Y_i \not = 0\}.
\]
The {\bf composition} $X\circ Y$ of two signed sets is a signed set defined by
\[
(X\circ Y)_i :=
\begin{cases}
X_i &\text{if } X_i \not=0 \\
Y_i &\text{otherwise}
\end{cases}
\qquad \text{for all }i\in \cI.
\]
A {\bf conditional oriented matroid}
on the ground set $\cI$ is a collection $\cL$ of signed sets, called {\bf covectors},
satisfying both of the following two conditions:
\begin{itemize}
\item {\sf Face Symmetry (FS):} If $X,Y\in \cL$, then $X \circ -Y \in \cL$.
\item{\sf Strong Elimination (SE):}  If $X,Y\in \cL$ and $i\in \Sep(X,Y)$, then there exists $Z\in \cL$ with
$Z_i  = 0$ and $Z_j  = (X\circ Y)_{j}$ for all $j \in \cI \setminus \Sep(X,Y)$.
\end{itemize}
If $\cL$ also contains the empty signed set $(\emptyset,\emptyset)$, then $\cL$ is an {\bf oriented matroid}.
We defer the key example to Example \ref{basic ex} while we make a few more definitions; the reader is invited
to skip ahead for motivation.
\excise{
Recall that the usual oriented matroid axioms replace {\sf Face Symmetry (FS)} with a symmetry axiom (if $X\in \cL$, then $-X\in\cL$) and a composition axiom (if $X,Y\in \cL$, then $X\circ Y$ is in $\cL$).
The composition axiom is implied from {\sf (FS)} and thus holds for all conditional oriented matroids; see Remark \ref{without the minus} (below).
When we moreover have $(\emptyset,\emptyset)\in\cL$, then {\sf(FS)} implies that $(\emptyset,\emptyset) \circ - X = - X$ is in $\cL$ for all $X\in\cL$.
We refer to the signed sets $X\in\cL$ as {\bf covectors} of the conditional oriented matroid, as they generalize the covectors of an oriented matroid; see Example \ref{basic ex} (below).}

\begin{remark}\label{without the minus}
The face symmetry condition also implies that $\cL$ is closed under composition, as
$$X\circ Y = (X \circ -X)\circ Y = X \circ (-X \circ Y) = X \circ -(X \circ -Y).$$
\end{remark}

\begin{remark}
The definition of conditional oriented matroid in \cite{bck2018}
includes the additional hypotheses that $\cI$ and $\cL$ are both nonempty.
We omit these hypotheses, both so that Example \ref{basic ex} always makes sense even when $\cA$ or $\cK$ is empty,
and so that  deletion and contraction are always defined (see Section \ref{sec:minors}).
\end{remark}

Let $\tau\subset\cL$ be the set of covectors that are nonzero in every coordinate.
Note that, if there is an element $i\in\cI$ such that $X_i=0$ for all $X\in\cL$ (such an $i$ is called a {\bf coloop}),
then $\tau=\emptyset$.
If there are no coloops, then elements of $\tau$ are called {\bf topes}.
We define the {\bf Gelfand--Rybnikov ring} $\GR(\cL)$ to be the ring of functions
from $\tau$ to $\Z$.
For each element $i\in \cI$, we define the {\bf Heaviside functions} $h_i^\pm\in \GR(\cL)$ by
\[
h_i^+(X) = \begin{cases} 1 & \text{if $X_i=+$}\\ 0 & \text{if $X_i=-$}\end{cases}
\and
h_i^-(X) = 1 - h_i^+(X) = \begin{cases} 1 & \text{if $X_i=-$}\\ 0 & \text{if $X_i=+$.}\end{cases}
\]
These functions generate the ring $\GR(\cL)$, and we define a filtration by letting $F_k(\cL)\subset \GR(\cL)$
be the subgroup generated by polynomial expressions in the Heaviside functions of degree at most $k$.

In Theorem \ref{thm:presentations}, the generators of our rings will be the images of the Heaviside functions,
and the relations will be indexed by circuits.
The notion of a circuit of a conditional oriented matroid does not appear in \cite{bck2018}, so we introduce it here.
A signed set $X$ is called a {\bf circuit} of $\cL$ if the following two conditions hold:
\begin{itemize}
\item For every covector $Y\in\cL$, $X\circ Y\neq Y$.
\item The signed set $X$ is support-minimal with respect to this property.  That is, if $Z$ is a signed set with $\underline{Z}\subsetneq\underline{X}$, then there is some $Y\in\cL$ with $Z\circ Y = Y$.
\end{itemize}
We denote the set of circuits by $\cC$.
When $\cL$ is an oriented matroid, then this set agrees with the usual notion of circuits for oriented matroids (see Lemma \ref{lemma:circuits}).

\begin{example}\label{basic ex}
Let $(\cA,\cK)$ be a pair consisting of an affine hyperplane arrangement $\cA$ in a real vector space $V$
and a convex open subset $\cK\subset V$.
Fix in addition a co-orientation of each $H\in V$, so that we can talk about the positive open half space $H^+$
and the negative open half space $H^-$, with $V = H^+ \sqcup H^- \sqcup H$.  For any signed set $X$ in $\cA$,
let $$H_X := \bigcap_{H\in X^+} H^+ \cap \bigcap_{H\in X^-} H^- \cap \bigcap_{H \in \cA\setminus\underline{X}} H.$$
We then define $$\cL(\cA,\cK) := \{X\mid H_X \cap \cK \neq \emptyset\},$$
and observe that $\cL(\cA,\cK)$ is a conditional oriented matroid on $\cA$.  The face symmetry property
comes from the fact that $\cK$ is open, and the strong elimination property comes from the fact that $\cK$ is convex.
Each point $p\in \cK$ determines a covector $X\in \cL(\cA,\cK)$ by putting
$X_H = \pm$ if $p\in H^\pm$ and $X_H = 0$ if $p\in H$, and every covector arises in this manner.
The conditional oriented matroid $\cL(\cA,\cK)$ is an oriented matroid if and only if there is a point that lies in every hyperplane
as well as in $\cK$, in which case $\cL(\cA,\cK) = \cL(\cA,V)$.

The conditional oriented matroid $\cL(\cA,\cK)$ has no coloops, the topes correspond to the
connected components of $M_1(\cA,\cK)$, and therefore the
the Gelfand-Rybnikov ring of $\cL(\cA,\cK)$ coincides, as a filtered ring,
with the Varchenko--Gelfand ring of $(\cA,\cK)$.
The circuits of $\cL(\cA,\cK)$ are the minimal signed sets $X$ with the property that
$$\bigcap_{H\in X^+} H^+ \cap \bigcap_{H\in X^-} H^- \;\cap\; \cK = \emptyset.$$
An explicit example of this form appears in Example \ref{explicit}.
\end{example}

We are now ready to give our presentations.  Consider the free graded $\Z[u]$-algebra
\[
R := \Z\left[u,e_i^+,e_i^-\right]_{i\in\cI}\Big{/}\left\langle e_i^+e_i^-, e_i^+ + e_i^- - u\bigmid i\in\cI\right\rangle\,,
\]
with all generators having degree 2.
For each signed set $X$, let
\[
e_X := \prod_{i\in X^+} e_i^+ \prod_{i\in X^-} (-e_i^-)\in R\and f_X := \frac{e_X - e_{-X}}{u}\in R\,,
\]
and consider the ideals
\[
I_{\cL} := \big\langle e_X\mid X\in\cC\big\rangle\subset R
\and J_{\cL} := \big\langle f_X \mid \pm X\in\cC \big\rangle \subset R\,.
\]
For $m\in \{0,1\}$, consider the quotient ring
$R_{m} := R/\langle u-m\rangle$, and let $I_{\cL,m}$ and $J_{\cL,m}$ be the images of $I_{\cL}$ and $J_{\cL}$ in $R_{m}$.


\begin{theorem}\label{thm:presentations}
We have canonical isomorphisms
\begin{align*}
\GR(\cL) &\cong R_{1}\Big{/}I_{\cL,1}+J_{\cL,1}\\
\gr \GR(\cL) &\cong R_{0}\Big{/}I_{\cL,0}+ J_{\cL,0}\\
\Rees \GR(\cL) &\cong R\Big{/}I_{\cL} + J_{\cL}
\end{align*}
given by sending each $e_i^{\pm}$ to the image of the corresponding Heaviside function $h_i^\pm$.
\end{theorem}

\begin{remark}
Theorem \ref{thm:presentations} has many antecedents.
When $\cL = \cL(\cA,V)$, it is due to Varchenko and Gelfand \cite{VG} (see also \cite{dLS,moseley} for the connections
to cohomology and equivariant cohomology, respectively).  When $\cL$ is an oriented matroid,
it is due to Gelfand and Rybnikov (see \cite{Co} for a study of the associated graded ring).
When $\cL = \cL(\cA,\cK)$ as in Example \ref{basic ex}, it is due to the first author \cite{db}.\footnote{With
the exception of \cite{moseley}, none of these previous works explicitly mention the Rees algebra, but the third isomorphism
can be derived from the other two.}
\end{remark}

\begin{remark}
The ideal $I_{\cL,1}+J_{\cL,1}$ is inhomogeneous, and it is clear that its initial ideal
contains $I_{\cL,0}+J_{\cL,0}$.  The fact that its initial ideal is equal to $I_{\cL,0}+J_{\cL,0}$ is not obvious;
the proof of this fact is a substantial part of the proof of Theorem \ref{thm:presentations}.
This is equivalent to the statement that $R\Big{/}I_{\cL} + J_{\cL}$ is a free module over $\Z[u]$.
\end{remark}

\begin{remark}
If $X_i=+$, then $e_i^+ \thinspace f_X = e_X$. 
For this reason, we may replace the ideal $I_{\cL}$ with the ideal
$$I'_{\cL} := \big\langle e_X\mid X\in\cC, -X\notin\cC\big\rangle$$
in the statement of Theorem \ref{thm:presentations}.  If $\cL$ is an oriented matroid, then $I'_{\cL} = 0$,
thus we can eliminate the ideals $I_{\cL}$ and $I_{\cL,m}$ entirely from the statement of the theorem.
This gives us the presentations appearing in \cite{VG,GelfandRybnikov,dLS,moseley}.
\end{remark}

\begin{remark}
The most difficult part proving Theorem \ref{thm:presentations} is developing the theory of circuits of conditional oriented matroids,
leading up to the proof of Proposition \ref{nbc-tope}.  This proposition has a relatively easy proof when $\cL = \cL(\cA,\cK)$
(see Remark \ref{easy case}), but the proof for general conditional oriented matroids is much more involved.
\end{remark}

\begin{figure}
    \centering
    \begin{tikzpicture}[bezier bounding box]
\draw[color=violet!80!white,fill=violet!20!white,dashed,thick]
    (0,0) .. controls +(0,3) and + (0,2) .. (5,1)
          .. controls +(0,-3) and + (0,-2) ..
          cycle;
\draw[black,thick]
    (-1.5,1) -- (6,1) node[right]{$3$} (5.5,1) node[above]{+}
    (-1.3,1.5) -- (5.7,-2) node[right]{4}
    (5.5,-1.8) node[above]{+}
    (1.5,3) node[above]{$1$} -- (4.5,-3)
    (1.7,2.6) node[left]{+}
    (3.5,3) node[above]{$2$} -- (0.5,-3)
    (3.5,3) node[left]{+}
    (4,0) node{\textcolor{violet}{$\cK$}}
    ;
\end{tikzpicture}
    \caption{Four co-oriented lines $1,2,3,4$ in the plane along with a convex open subset $\cK$. The co-orientation is indicated with a $+$ on the positive side of a given line.}
    \label{fig:four-lines}
\end{figure}
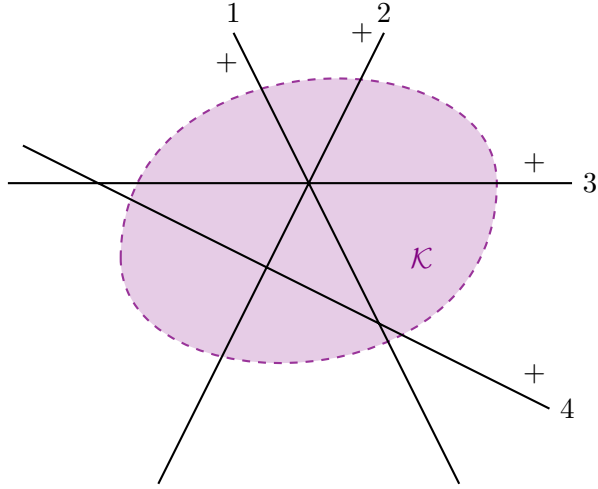

\begin{example}\label{explicit}
Figure \ref{fig:four-lines} shows an arrangement $\cA$ of four lines in the plane, along with a convex open subset $\cK$.
Example \ref{basic ex} tells us that we have
$\cC = \{\pm X,Y,Z\}$, where
\begin{eqnarray*}
  X = & (\{1,3\},\{2\}) & = (+,-,+,0)\\
  Y = & (\{3\},\{4\})   & = (0,0,+,-)\\
  Z = & (\{2,4\},\{1\}) & = (-,+,0,+).
\end{eqnarray*}
Theorems \ref{cohomology} and \ref{thm:presentations} imply that
$$\H^*_T(M_3(\cA,\cK); \Z) \cong \Rees\VG(\cA,\cK) = \Rees\GR(\cL(\cA,\cK))
\cong R/\langle f_X, e_Y, e_Z\rangle.$$
Explicitly, $R/\langle f_X, e_Y, e_Z\rangle$ is equal to
$$\Z[e_1^+,e_2^+,e_3^+,e_4^+,u]\Big{/}\Big\langle e_i^+(u-e_i^+), e_1^+e_2^+ - e_1^+e_3^+ + e_2^+e_3^+ - ue_2^+,
e_3^+(e_4^+-u), (e_1^+-u)e_3^+e_4^+\Big\rangle,$$
where the first relation $e_i^+(u-e_i^+)$ holds for $i\in\{1,2,3,4\}$.  The rings $\gr\VG(\cA,\cK)$ and $\VG(\cA,\cK)$
are obtained from $\Rees\VG(\cA,\cK)$ by setting $u$ equal to 0 and 1, respectively.
\end{example}

\vspace{\baselineskip}
\noindent
{\em Acknowledgments:}
The authors are grateful to the organizers of the {\em Arrangements at Home} workshop series for bringing us together to work on these problems,
and to 
Vic Reiner for helpful conversations.

\section{Background on Equivariant Cohomology}\label{sec:eq}
The main results of this section are Propositions \ref{eq stuff}, \ref{lifting generators}, and \ref{rees it up}, which are key steps in the proof of Theorem \ref{cohomology}.
The ideas in this section are not new, but we found it difficult to find precise statements in the literature about
equivariant cohomology with $\Z$ coefficients.  We collect the results that we need here,
and point the reader to any of \cite{Borel, AB, GKM} for a standard introduction to equivariant cohomology.
For simplicity, we will only discuss actions of the group $T := U(1)$.

\subsection{The definition}
Let $ET := \C^\infty\setminus\{0\}$.
This is a contractible space, equipped with a free action of $T$.
Define the quotient
$$BT := ET/T \cong \C P^{\infty}.$$
A {\bf \boldmath{$T$}-space} is a space equipped with a continuous action of the group $T$.
For any $T$-space $M$, the {\bf Borel space} of $M$ is
$$M_T := (M \times ET)/T,$$
where $T$ acts diagonally on $M\times ET$.
The {\bf equivariant cohomology} of $M$ is the graded ring
$$\H^*_T(M; \Z) := \H^*(M_T; \Z).$$
The $T$-equivariant projection from $M\times ET$ to $ET$ descends to a fiber bundle $\pi:M_T\to BT$ with fiber $M$.
Pulling back along $\pi$ makes $\H^*_T(M; \Z)$ into an algebra over $\H^*(BT; \Z) \cong \Z[u]$.
Any $T$-equivariant map from $M$ to another $T$-space $N$
induces a map from $M_T$ to $N_T$ that is compatible with the bundle projections.
In particular, this means that $\H^*_T(-; \Z)$ is a
contravariant functor
from the category of $T$-spaces with equivariant maps to the category of graded $\Z[u]$-algebras.
If $N\subset M$ is a $T$-subspace, we define the relative $\Z[u]$-modules
$\H^*_T(M,N;\Z) := \H^*(M_T,N_T;\Z).$

\begin{example}\label{ex:acts trivially}
If $T$ acts trivially on $M$, then $M_T \cong M \times BT$, and $$\H^*_T(M; \Z) \cong \H^*(M; \Z)\otimes \H^*(BT; \Z) \cong \H^*(M; \Z)\otimes \Z[u].$$
\end{example}

\begin{example}\label{free}
If $T$ acts freely on $M$, then $M_T \cong M/T \times ET$, which is homotopy equivalent to $M/T$, so $\H^*_T(M; \Z)\cong \H^*(M/T; \Z).$
More generally, if $N\subset M$ is a $T$-subspace, then $$\H^*_T(M,N; \Z)\cong \H^*(M/T, N/T; \Z).$$
\end{example}

\subsection{Specializations}\label{sec:specializations}

We now introduce two specialization homomorphisms $\bar\varphi$ and $\bar\psi$
that we will need for the proof of Theorem \ref{cohomology}.
The inclusion of a fiber $\iota:M\to M_T$ defines a graded algebra homomorphism $$\varphi := \iota^* :\H^*_T(M; \Z)\to \H^*(M; \Z),$$
called the {\bf forgetful homomorphism}.
By construction, the composition $\pi\circ\iota$ is constant, therefore
$$\varphi(\pi^*u) = \iota^*(\pi^* u) = (\pi\circ\iota)^*u = 0.$$
Since $\pi^*u$ lies in the kernel of $\varphi$, we have the induced homomorphism
$$\bar\varphi:\H^*_T(M; \Z)/\langle \pi^*u\rangle \to \H^*(M;\Z).$$
We will often abuse notation and use the symbol $u$ to denote $\pi^*u \in \H^*_T(M; \Z)$.
This allows us to rewrite the previous line as a specialization $u = 0$:
$$\bar\varphi:\H^*_T(M; \Z)/\langle u\rangle \to \H^*(M;\Z).$$

The inclusion of the fixed point set $\kappa:M^T\to M$ induces another graded algebra homomorphism
$$\psi := \kappa^*:\H^*_T(M; \Z)\to \H^*_T(M^T; \Z) \cong \H^*(M^T; \Z) \otimes \Z[u].$$
When we set $u$ equal to $1$, this descends to a homomorphism
$$\bar\psi:\H^*_T(M; \Z)/\langle u-1\rangle \to \H^*(M^T;\Z).$$

\subsection{Localization}


\begin{lemma}\label{torsion}
Let $M$ be a $T$-manifold of dimension $d$. If $T$ acts freely on $M\setminus M^T$, then the relative cohomology group $H_T^*(M, M^T; \Z)$ is
annihilated by $u^d$.  
\end{lemma}

\begin{proof}
Let $N$ be a $T$-equivariant closed tubular neighborhood of $M^T$ in $M$, and let $U$ be the interior of $N$.
Then $$H_T^*(M, M^T; \Z) \cong H_T^*(M,N;\Z) \cong H_T^*(M\setminus U,N\setminus U; \Z) \cong \H^*\!\Big((M\setminus U)/T,(N\setminus U)/T; \Z\Big),$$
where the first isomorphism is induced by the $T$-equivariant deformation retraction from $N$ to $M^T$,
the second by excision, and the third by Example \ref{free}.
Since $T$ acts freely away from $M^T$, $(M\setminus U)/T$ is a manifold with boundary $(N\setminus U)/T$.
The cohomology of this manifold vanishes in degrees greater than its dimension, and the lemma follows.
\end{proof}


\begin{corollary}\label{localization}
Let $M$ be a $T$-manifold of dimension $d$.
If $T$ acts freely on $M\setminus M^T$, then the kernel and cokernel of the map
$$\psi:\H^*_T(M; \Z) \to \H_T^*(M^T;\Z)$$
are annihilated by $u^d$.
\end{corollary}

\begin{proof}
Consider the long exact sequence in equivariant cohomology associated with the pair $(M,M^T)$.
This shows that the kernel (respectively cokernel) of $\psi$ is a submodule (respectively quotient module) of $H^*_T(M,M^T;\Z)$,
thus the statement follows from Lemma \ref{torsion}.
\end{proof}

\begin{remark}
If we drop the assumption that $T$ acts freely on $M\setminus M^T$, then Corollary \ref{localization} holds over $\Q$,
but not over $\Z$.  For example, consider the action of $T$ on itself as multiplication by the square, so that the fixed locus is empty and
the stabilizer of every point is $\{\pm 1\}\subset T$.
Then $$\Ker(\psi) = \H^*_T(T; \Z) \cong \H^*(B\{\pm 1\}; \Z) \cong \H^*(\R P^\infty; \Z) \cong \Z[u]/\langle 2u\rangle.$$
Only after tensoring with $\Q$ is this module annihilated by $u$.
\end{remark}

\subsection{Equivariant formality}\label{sec:formailty}
We say that the $T$-space $M$ is {\bf equivariantly formal over \boldmath{$\Z$}} if the map $\varphi:\H_T^*(M; \Z)\to \H^*(M; \Z)$ is surjective.

\begin{proposition}\label{eq stuff}
If $M$ is equivariantly formal over $\Z$, then
$\H^*_T(M; \Z)$ is free as a $\Z[u]$-module, and
the specialization $\bar\varphi:\H^*_T(M; \Z)/\langle u\rangle \to \H^*(M;\Z)$ is an isomorphism.
\end{proposition}

\begin{proof}
Apply the Leray--Hirsch theorem (see \cite[Theorem 4D.1]{hatcher}, for example) to the fiber bundle $\pi:M_T\to BT$.
\end{proof}

\begin{proposition}\label{lifting generators}
If $M$ is equivariantly formal over $\Z$ and $x_1,\ldots,x_n\in \H^*_T(M;\Z)$ are homogeneous classes with the property that $\varphi(x_1),\ldots,\varphi(x_n)$
generate $\H^*(M;\Z)$ as a ring, then $x_1,\ldots,x_n$ generate $H_T^*(M;\Z)$ as a $\Z[u]$-algebra.
\end{proposition}

\begin{proof}
Let $R\subset \H^*_T(M;\Z)$ be the subalgebra generated by the classes $x_1,\ldots,x_n$.
Assume for the sake of contradiction that $R\subsetneq\H^*_T(M;\Z)$, and let $\alpha\in \H^*_T(M;\Z)\setminus R$ be a homogeneous class
of minimal degree.  Since $\varphi(R)= \H^*(M;\Z)$, there exists a class $x\in R$ such that $\varphi(x) = \varphi(\alpha)$.
Since $\varphi$ is a homomorphism
\[
0 = \varphi(x) - \varphi(\alpha) = \varphi(x-\alpha).
\]
By the second part of Proposition \ref{eq stuff}, the kernel of $\varphi$ is generated by $u$, so there is a class $\beta$ such that $x-\alpha = u\beta$.
In particular, this $\beta$ has degree strictly less than $\alpha$.
Since $\alpha$ had minimal degree, this means that $\beta\in R$, which contradicts the assumption that $\alpha\notin R$.
\end{proof}

\begin{proposition}\label{all together}
Suppose that $M$ is a finite dimensional $T$-manifold with the property that $T$ acts freely on $M\setminus M^T$,
and that $M$ is equivariantly formal over $\Z$.
Then $$\psi:\H^*_T(M; \Z) \to \H_T^*(M^T;\Z)$$ is injective and $$\bar\psi:\H^*_T(M; \Z)/\langle u-1\rangle \to \H^*(M^T;\Z)$$ is an isomorphism.
\end{proposition}

\begin{proof}
Corollary \ref{localization} tells us that the kernel of $\psi$ is annihilated by a power of $u$.
Proposition \ref{eq stuff} tells us that the domain of $\psi$ is a free $\Z[u]$-module, thus the kernel of $\psi$ must be trivial.

For the second statement, recall from Example \ref{ex:acts trivially} that $\H^*_T(M^T;\Z) \cong \H^*(M^T;\Z)\otimes\Z[u]$, and consider the short exact sequence
$$0\to \H^*_T(M; \Z)\overset{\psi}{\longrightarrow} \H^*(M^T;\Z)\otimes\Z[u]\to\Coker(\psi)\to 0.$$
Taking the tensor product over $\Z[u]$ with $Q := \Z[u]/\langle u-1\rangle$, we obtain the exact sequence
$$\Tor_1(\Coker(\psi), Q)\to \H^*_T(M; \Z)/\langle u-1\rangle\overset{\bar\psi}{\longrightarrow} \H^*(M^T;\Z)\to \Coker(\psi)\otimes_{\Z[u]} Q\to 0.$$
Corollary \ref{localization} implies that $\Coker(\psi)\otimes_{\Z[u]} Q$ is zero.
From the definition of $Q$ and the fact that $\Coker(\psi)$ is graded, we have $$\Tor_1(\Coker(\psi), Q)\cong \{x\in\Coker(\psi)\mid (u-1)\thinspace x=0\}=0.$$
This completes the proof that $\bar\psi$ is an isomorphism.
\end{proof}

\subsection{The equivariant filtration}
In this section we suppose that $M$ is a finite dimensional $T$-manifold with the property that $T$ acts freely on $M\setminus M^T$,
and that $M$ is equivariantly formal over $\Z$, so that $\bar\psi$ is an isomorphism (by Proposition \ref{all together}).
Moreover, we assumee that the cohomology for $M$ vanishes in odd degree, and note that Proposition \ref{eq stuff} implies that the same is true for equivariant cohomology.
In this setting $\H^*(M^T; \Z)$ admits an interesting filtration, which we describe now.

For $k\geq 1$, define $$F_k(M)\subset \H^*(M^T; \Z) \cong \H^*_T(M; \Z)/\langle u-1\rangle$$
to be the set of classes that can be lifted to $\H^{2k}_T(M;\Z)$.
Note that any class which can be lifted to $\a\in \H^{2k}_T(M; \Z)$ can also be lifted
to $u^i\a\in \H^{2(k+i)}_T(M; \Z)$ for $i\geq 0$.
Thus the groups $F_k(M)$ form a filtration
$$F_0(M) \subset F_1(M)\subset\cdots \subset \H^*(M^T; \Z),$$
which we call the {\bf equivariant filtration}.  The following proposition is immediate from the definition of the equivariant filtration.

\begin{proposition}\label{rees it up}
If $M$ satisfies the hypotheses of Proposition \ref{all together} and has vanishing odd cohomology, then
the image of the inclusion $$\psi:\H_T^*(M; \Z)\to \H^*_T(M^T; \Z)
$$
is the Rees algebra of the equivariant filtration,
and therefore $\psi$ induces an isomorphism
$$\H_T^*(M; \Z)\cong \Rees \H^*(M^T; \Z)$$
of graded $\Z[u]$-algebras.
\end{proposition}

\begin{remark}
If one wants to drop the assumption that the odd cohomology vanishes, one can alternatively define a filtration of $\H^*(M^T; \Z)$
by taking the $k^\text{th}$ filtered piece to be the images of classes of degree $\leq k$ (rather than $2k$) in $\H^*_T(M; \Z)$.
The Rees algebra of this filtration will be isomorphic to the algebra $\H^*_T(M; \Z)\otimes_{\Z[u]} \Z[u^{\nicefrac{1}{2}}]$,
where now the Rees parameter corresponds to $u^{\nicefrac{1}{2}}$ rather than $u$.
\end{remark}

\subsection{Classes represented by submanifolds}\label{sec:submanifolds}

We have now collected the key general results needed for the proof of Theorem \ref{cohomology}.
In this section, we construct a family of equivariant cohomology classes and state some of their properties.
Throughout this section, suppose that $M$ is a manifold and $N\subset M$ is a closed\footnote{We mean closed in the sense that $N$
is a closed subset of $M$.  We do not mean to imply that $N$ is compact, which is what topologists sometimes mean by the phrase ``closed manifold".} submanifold of codimension $k$.
A coorientation of $N$ (= a choice of orientation of the normal bundle) determines a cohomology class $[N]\in \H^k(M; \Z)$.
One construction of this class is as follows.  Let $U$ be a tubular neighborhood of $N$ in $M$, and let $\bar U := M/(M\setminus U)$.
Then $\bar U$ is isomorphic to the Thom space of the normal bundle to $N$, and we therefore have the Thom isomorphism
$\H^*(N;\Z)\cong \H^{*+k}(\bar U;\Z)$.
The class $[N]$ is obtained by pulling back the class $1\in \H^0(N;\Z)\cong \H^k(\bar U;\Z)$ to $M$.

If $T$ acts on both $M$ and $N$, then we also define
$$[N]_T = [N_T] \in \H^k(M_T; \Z) = \H^k_T(M; \Z).$$
This construction has the following properties.
The first four follow from the corresponding non-equivariant statements, while (v) follows from (iv) applied to the inclusion of $M$ into $M_T$.
\begin{itemize}
\item[(i)] Reversing the coorientation of a submanifold $N$ sends $[N]_T$ to $-[N]_T$.
\item[(ii)] If $N_1$ and $N_2$ are disjoint closed cooriented $T$-submanifolds, then $[N_1\cup N_2]_T = [N_1]_T + [N_2]_T$.
\item[(iii)] If $N_1$ and $N_2$ are transverse closed cooriented  $T$-submanifolds, then $[N_1\cap N_2]_T = [N_1]_T \cdot [N_2]_T$.
\item[(iv)] If $N$ is a cooriented closed $T$-submanifold of $M$ and $f:M'\to M$ is a $T$-equivariant map that is transverse to $N$,
then $f^*([N]_T) = [f^{-1}N]_T \in \H^*(M'; \Z)$.
\item[(v)] This construction is compatible with the forgetful homomorphism.  That is, $\varphi([N]_T) = [N]$.
\end{itemize}

\begin{example}
Let $M=\C$ with the standard action of $T$.  Then $M_T$ is homotopy equivalent to $BT$, so $\H^*_T(M; \Z) \cong \H^*(BT; \Z) \cong \Z[u]$.
The class $[\{0\}]_T\in \H^2_T(M; \Z)$ is a generator, and can therefore be identified with $u$.
\end{example}

\begin{example}\label{base}
Let $M = \R^3\setminus\{0\} = \R\times\C\setminus\{(0,0)\}$.
Let $$e^+ := [\R_{>0}\times\{0\}]_T\in \H^2_T(M; \Z)\and e^- := [\R_{<0}\times\{0\}]_T\in \H^2_T(M; \Z).$$
Note that $\varphi(e^+) = -\varphi(e^-)$ generates $\H^2(M; \Z)$, so $M$ is equivariantly formal over $\Z$.
Since $e^+$ and $e^-$ are represented by disjoint submanifolds, we have $e^+e^- = 0$ by (iii).

Consider the projection $f:M\to \C$.  Since $f$ is transverse to the submanifold $\{0\}$, we have
the following equalities in $\H^2_T(M; \Z)$
\begin{align*}
u & = f^*u = f^*([\{0\}]_T) = [f^{-1}(0)]_T\\
& = [\R_{>0}\times\{0\} \ \cup\ \R_{<0}\times\{0\}]_T = [\R_{>0}\times\{0\}]_T + [\R_{<0}\times\{0\}]_T\\
& = e^+ +\ e^-
\end{align*}
where the equality on the middle line comes from $(ii)$ and the fact that the two manifolds are disjoint.
In particular, our discussion from Section \ref{sec:specializations} implies that the class $e^+ +\ e^-$ is in the kernel of $\varphi$ and from Proposition \ref{eq stuff} we have
$$\H^*_T(M; \Z) \cong \Z[e^+, e^-]/\langle e^+e^-\rangle.$$
\end{example}

The following lemma is not strictly necessary for proving our results, and will be referenced only in Remark \ref{proof using empty intersections}.
That said, it is a fundamental property of equivariant cohomology which we believe is helpful for understanding the ideas in this paper.

\begin{lemma}\label{empty intersection}
If $N_1,\ldots,N_r$ are (not necessarily transverse) closed cooriented $T$-submanifolds of $M$
with $N_1\cap\cdots\cap N_r=\emptyset$, then $[N_1]_T\cdots [N_r]_T=0$.
\end{lemma}

\begin{proof}
The class $[N_1]_T\cdots [N_r]_T$ is equal to the image of the class $[N_1]_T\otimes \cdots\otimes  [N_r]_T$ under the composition
$$\H^*_T(M;\Z)\otimes\cdots\otimes \H^*_T(M;\Z)\to \H^*_T(M^r; \Z)\overset{\Delta^*}{\longrightarrow} \H^*_T(M;\Z),$$
where $M^r$ is the direct sum of $r$ copies of $M$ and $\Delta$ is the diagonal map.
Let
\[
\bar N_i := \{(p_1,\ldots,p_r)\in M^r\mid p_i \in N_i\}.
\]
The image of $[N_1]_T\otimes \cdots\otimes  [N_r]_T$ in $\H^*_T(M^r; \Z)$ is equal to the product of the classes $[\bar N_i]_T$.
Since the submanifolds $\bar N_i$ are pairwise transverse, we can apply $(iii)$ to this product.
The product is zero because the total intersection is empty.
\end{proof}

\section{Proof of Theorem \ref{cohomology}}\label{sec:cohomology}
In this section, we will prove Theorem \ref{cohomology}.
The key technical ingredient will be the statement that $M_3(\cA,\cK)$ is equivariantly formal over $\Z$ (Proposition \ref{formal}).

Let $V$ be a finite dimensional real vector space, $\cA$ a finite set of affine hyperplanes in $V$,
and $\cK\subset V$ a convex open set.
For any $H\in\cA$, 
we define the {\bf deletion} $(\cA',\cK')$, which is a pair consisting of an arrangement and a convex open subset in $V$, by setting
$\cA' = \cA\setminus \{H\}$ and $\cK' = \cK$.  We also define the {\bf contraction}
$(\cA'',\cK'')$, which is a pair consisting of an arrangement and a convex open subset in $H$, by setting
$\cA'' = \{H'\cap H\mid H'\in\cA'\}$ and $\cK'' = \cK \cap H$.
Note that we have an open inclusion from $M_3(\cA,\cK)$ to $M_3(\cA',\cK')$, and the complement is equal to $M_3(\cA'',\cK'')$.
The following lemma is standard in the case where $\cK=V$,
and the proof in this more general setting is identical.

\begin{lemma}\label{relative}
We have a canonical isomorphism of graded abelian groups
$$\H^{i}(M_3(\cA',\cK'), M_3(\cA,\cK); \Z) \cong \H^{i-2}(M_3(\cA'',\cK''); \Z).$$
\end{lemma}

\begin{proof}
The key observation is that the normal bundle to $M_3(\cA'',\cK'')$ inside of $M_3(\cA',\cK')$ is a trivial bundle of rank 3.
By the Tubluar Neighborhood Theorem, the Excision Theorem, and the K\"unneth Theorem, this implies that
\begin{eqnarray*} \H^{i}(M_3(\cA',\cK'), M_3(\cA,\cK); \Z)
&\cong& \H^{i-2}(M_3(\cA'',\cK''); \Z) \otimes \H^2(\R^3,\R^3\setminus\{0\}; \Z)\\
&\cong& \H^{i-2}(M_3(\cA'',\cK''); \Z).\end{eqnarray*}
This completes the proof.
\end{proof}

\begin{corollary}\label{two sequences}
We have a map of long exact sequences
\begin{center}
  \begin{tikzcd}[column sep = 0.4cm]
 \cdots \ar[r] & \H^i(M_3(\cA'); \Z) \ar[r]\ar[d] & \H^i(M_3(\cA); \Z) \ar[r]\ar[d] & \H^{i-2}(M_3(\cA'');\Z) \ar[r]\ar[d] & \H^{i+1}(M_3(\cA');\Z)\ar[r]\ar[d] & \cdots \\
 \cdots \ar[r] & \H^i(M_3(\cA',\cK'); \Z) \ar[r] & \H^i(M_3(\cA,\cK); \Z) \ar[r] & \H^{i-2}(M_3(\cA'',\cK'');\Z) \ar[r] & \H^{i+1}(M_3(\cA',\cK');\Z)\ar[r] & \cdots.
  \end{tikzcd}
\end{center}
\end{corollary}

\begin{proof}
The bottom row is obtained by applying Lemma \ref{relative} to the long exact sequence of the pair $(M_3(\cA',\cK'), M_3(\cA,\cK))$.
The top row comes from the pair $(M_3(\cA'), M_3(\cA))$, which is the special case where $\cK = V$.
The vertical maps are restriction maps.
\end{proof}

\begin{proposition}\label{surj}
The restriction map $\H^*(M_3(\cA); \Z)\to \H^*(M_3(\cA,\cK); \Z)$ is surjective.
\end{proposition}

\begin{proof}
We proceed by induction on the cardinality of $\cA$.  The base case is where $\cA$ is empty, in which case the result follows from the fact that $\cK$ is convex.
If there is at least one hyperplane, then we may apply Corollary \ref{two sequences},
and our inductive hypothesis tells us that two thirds of the downward arrows are surjective (those concerning $\cA'$ and $\cA''$).  The fact that the remaining downward arrows are surjective
follows from a diagram chase.
\end{proof}

For each hyperplane $H\in \cA$, let $f_H:V\to \R$ be an affine linear form with vanishing locus $H$, and consider the induced map
$$g_H:M_3(\cA,\cK)\to \R^3\setminus\{0\}.$$ 
Let $$e_H^+:= g_H^*(e^+)\and e_H^- := g_H^*(e^-),$$
where $e^+,e^-\in \H^2_T(\R^3\setminus\{0\}; \Z)$ are defined in Example \ref{base}.

\begin{remark}
We note that $f_H$ induces a co-orientation of $H$, and the classes $e_H^{\pm}\in \H^2_T(M_3(\cA,\cK); \Z)$  will eventually be identified with the images of
the classes $e_H^\pm\in R_\cA$ via the isomorphisms of Theorems \ref{cohomology} and \ref{thm:presentations}.
\end{remark}


\begin{proposition}\label{formal}
The $T$-space $M_3(\cA,\cK)$ is equivariantly formal over $\Z$, and $\H^*_T(M_3(\cA,\cK);\Z)$ is generated as a $\Z[u]$-algebra
by the classes $\{e_H^\pm\mid H\in\cA\}$.
\end{proposition}

\begin{proof}
  In the setting where $\cK = V$ and all hyperplanes go through the origin, \cite[Corollary 5.6]{dLS} proved that $\H^*(M_3(\cA);\Z)$ is generated as a ring by the classes $\{\varphi(e_H^\pm)\mid H\in\cA\}$.
  It was extended\footnote{Moseley states his result with $\Q$ coefficients
  rather than $\Z$ coefficients,
  but his proof, which employs a deletion/contraction induction similar to the one that we use for Proposition \ref{surj},
  goes through unchanged.} to affine arrangements $\cA$ and $\cK = V$ in \cite[Lemma 4.2]{moseley}.
 Proposition \ref{surj} implies that the classes $\{\varphi(e_H^\pm)\mid H\in\cA\}$ also generate $\H^*(M_3(\cA,\cK);\Z)$.
  The result now follows from Proposition \ref{lifting generators}.
\end{proof}

\begin{proof}[Proof of Theorem \ref{cohomology}]
The first isomorphism follows from the definition of $\VG(\cA,\cK)$ and the fact that
$M_3(\cA,\cK)^T \cong M_1(\cA,\cK)$.
Propositions \ref{rees it up} and \ref{formal} imply that we have an isomorphism of graded $\Z[u]$-algebras
$$\H^*_T(M_3(\cA,\cK);\Z)\cong\Rees \H^*(M_1(\cA,\cK); \Z),$$
and Proposition \ref{eq stuff} tells us that setting $u=0$ gives the isomorphism
$$\H^*(M_3(\cA,\cK); \Z)\cong \gr \H^*(M_1(\cA,\cK); \Z).$$
The one subtlety is that the isomorphism coming from Proposition \ref{rees it up} involves the Rees algebra of the equivariant
filtration, and Theorem \ref{cohomology} is about the Heaviside filtration.  Thus we need to check that these two
filtrations coincide.

By definition, the $k^\text{th}$ piece of the Heaviside filtration consists of classes that can be expressed as polynomials of degree at most $k$ in the Heaviside functions.  On the other hand, the second half of Proposition \ref{formal} says that the $k^\text{th}$ piece of the equivariant filtration
consists of classes that can be expressed as polynomials of degree at most $k$ in the restrictions of $\{e_H^\pm\mid H\in\cA\}$ to $M_1(\cA,\cK)$, with $u$ specialized to 1.
It is therefore sufficient to observe that the restriction of $e_H^{\pm}$ is precisely the Heaviside function that takes the value 1 on $H^\pm$ and 0 on $H^\mp$.
\end{proof}

\section{Conditional oriented matroids}\label{sec:COMs}
The main result of this section is Proposition \ref{nbc-tope}, which is the key ingredient in the proof of Theorem \ref{thm:presentations}.
Proposition \ref{nbc-tope} extends a standard result for oriented matroids \cite[Exercise 4.46]{blvs} to the setting of conditional oriented matroids.
We first state the proposition and then develop the necessary theory to prove it.

Throughout this section, let $\cL$ be a conditional oriented matroid on the ground set $\cI$.
Fix a linear ordering $<$ on $\cI$, so that the support of every nonzero signed set $X$ has a unique minimum element, $\min(X)$.
For any signed set $X$ with nonempty support, consider the unsigned set
\[
\becircled{X} := \underline{X} \setminus \{\min(X)\}.
\]
As in the introduction, we denote the set of circuits of $\cL$ by $\cC$.
We call an unsigned set $S\subset\cI$ an {\bf NBC set} if it satisfies the following two conditions:
\begin{itemize}
\item If $X\in\cC$ is a circuit, then $\underline{X}\not\subset S$.
\item If $\pm X\in\cC$ are nonzero\footnote{The ``zero'' or ``empty'' signed set $(\emptyset,\emptyset)$ is a circuit if and only if $\cL=\emptyset$.} circuits, then $\becircled{X}\not\subset S$.
\end{itemize}
We denote the collection of all NBC sets by $\cN$.

\begin{remark}
If $\cL$ is an oriented matroid, then the first condition is redundant because $\cC$ is closed under negation and does not contain the empty signed set.
The set $\becircled{X}$ is called a {\bf broken circuit}, and NBC
stands for {\bf no broken circuit}.
\end{remark}

The following proposition relates the number of NBC sets to the number of covectors that are nonzero in every coordinate;
the remainder of this section will be devoted to its proof.

\begin{proposition}\label{nbc-tope}
The cardinality of $\cN$ is equal to the cardinality of $\cT$.
\end{proposition}

\begin{remark}\label{zero case}
Recall that a coloop is an element $i\in\cI$ such that $X_i=0$ for all $X\in\cL$.
If there exists a coloop, then $\cT$ is clearly empty.
Similarly, if $i$ is a coloop, then either the empty signed set is a circuit, or
the two signed sets $\pm X$ with support $\{i\}$ are circuits, so $\cN$ is also empty.
If there are no coloops, then $\cT$ is the set of topes, and Proposition \ref{nbc-tope} says that the number of NBC sets is equal to the number of topes.
\end{remark}


\begin{remark}\label{easy case}
Proposition \ref{nbc-tope} has a short proof in the setting of Example \ref{basic ex}, where $\cL = \cL(\cA,\cK)$.
Consider the poset of flats of $\cA$ whose intersection with $\cK$ are nonempty,
ordered by reverse inclusion, and let $\mu$ be the M\"obius function
on that poset.  Zaslavsky proves that $|\cT|$ is equal to
$\sum_F|\mu(V,F)|$ \cite[Theorem 3.2(A) and Example A]{zaslavsky77}.
Since the lower interval $[V,F]$ is a geometric lattice for any $F$ in our poset,
a theorem of Rota \cite[Theorem 1.1]{Sagan} says that $|\mu(V,F)|$ is equal to the number of NBC sets whose closure is $F$.
Taking the sum over all $F$, we obtain $|\cN|$.
\end{remark}


\subsection{Circuits}\label{sec:circuits}
In some of our arguments, we will reduce to the setting of oriented matroids and then appeal to the extensive literature for the results that we need.
In order to do that, we first need to show that our notion of circuits agrees with the established notion for oriented matroids.

Recall that a circuit of a conditional oriented matroid $\cL$ is a support-minimal signed set $X$ satisfying $X\circ Y\neq Y$ for all $Y\in\cL$.
We use $\cG$ to denote the set of (not necessarily support-minimal) signed sets satisfying the same condition:
\[
\cG := \{X\mid X\circ Y\neq Y \text{ for all } Y\in\cL\}.
\]

Given two signed sets $X$ and $Y$, we say that they are {\bf orthogonal} and write $X\perp Y$ if either $\underline{X}\cap\underline{Y}=\emptyset$,
or there exist $i,j\in\underline{X}\cap\underline{Y}$ such that $X_i=Y_i$ and $X_j=-Y_j$.
Let
\[
\tilde\cG := \{X\mid \text{$X\neq (\emptyset,\emptyset)$ and $X\perp Y$ for all $Y\in\cL$}\},
\]
and let $\tilde\cC$ be the set of support-minimal elements of $\tilde\cG$.
When $\cL$ is an oriented matroid
(that is, when $(\emptyset,\emptyset)\in\cL$), a circuit is defined to be an element of $\tilde\cC$.

\begin{lemma}\label{lemma:circuits}
If $\cL$ is an oriented matroid, then $\cC=\tilde\cC$.
\end{lemma}

\begin{proof}
We will show that $\tilde\cG\subset\cG$ and that $\cC\subset\tilde\cG$.
These two statements will imply the lemma.
Suppose that $X\in \tilde\cG$ and $Y\in\cL$.
If $\underline{X}\cap\underline{Y}=\emptyset$, then $X\circ Y \neq Y$.
Otherwise, there is some $j\in\underline{X}\cap\underline{Y}$ with $X_j = -Y_j$, which again implies that $X\circ Y\neq Y$.
Thus $\tilde\cG\subset\cG$.

Now suppose that $X\in\cC$.
Since $\cL$ is an oriented matroid, $X\neq(\emptyset,\emptyset)$.
Assume for the sake of contradiction that there is a covector $Y\in\cL$ that is not orthogonal to $X$.
This means that $\underline{X}\cap\underline{Y}\neq\emptyset$ and we either have $X_i = Y_i$ for all
$i\in\underline{X}\cap\underline{Y}$ or $X_i = -Y_i$ for all
$i\in\underline{X}\cap\underline{Y}$.  We may assume the former (otherwise we can replace $Y$ with $-Y$).
Define a new signed set $Z$ by putting
$Z_i = X_i$ for all $i\in\underline{X}\setminus\underline{Y}$ and $Z_j = 0$ for all other $j$.  Then $\underline{Z}\subsetneq\underline{X}$, so there exists a covector $W\in\cL$ such that $Z\circ W = W$.
However, this implies that $X\circ (Y\circ W) = Y\circ W$, contradicting the fact that $X\in\cG$.
\end{proof}

\begin{remark}\label{rem:cG-too-big}
Even when $\cL$ is an oriented matroid, the sets $\cG$ and $\tilde\cG$ need not be equal.
In particular, $\cG$ is upwardly closed:  if $X\in\cG$ and $X\circ Y = Y$, then $Y\in\cG$.
The same need not be true of $\tilde\cG$, which is the set of nonzero vectors of the oriented matroid $\cL$.
\end{remark}

\subsection{Deletions and contractions}\label{sec:minors}
In this section, we review the definitions of deletion and contraction for conditional oriented matroids, both of which were introduced in \cite{bck2018}.
We then state and prove some results about how circuits behave under these operations.

Fix an element $i\in\cI$.
For any signed set $X$ on $\cI$, we define $\pi(X)$ to be the signed set on $\cI\setminus\{i\}$ obtained by forgetting the $i^\text{th}$ coordinate:
$\pi(X)_j = X_j$ for all $j\neq i$.
The {\bf deletion} of $\cL$ at $i$ is
\[
\cL' := \{\pi(X)\mid X\in\cL\}.
\]
On the other hand, for any signed set $X''$ on $\cI\setminus\{i\}$, we define $\iota(X'')$
to be the signed set on $\cI$ obtained by extending by zero:  $\iota(X)_j = X''_j$ for all $j\neq i$ and $\iota(X'')_i=0$.
The {\bf contraction} of $\cL$ at $i$ is
$$\cL'' := \{X''\mid \iota(X'')\in \cL\} = \{\pi(X)\mid X\in\cL\,\text{ and }\,X_i = 0\}.$$
The deletion and the contraction are both themselves conditional oriented matroids on the ground set $I\setminus\{i\}$ \cite[Lemma 1]{bck2018}.

\begin{example}
f $\cL = \cL(\cA,\cK)$ and $H\in\cA$ as in Example \ref{basic ex}, then the deletion $\cL'$ is equal to $\cL(\cA\setminus\{H\},\cK)$.
The contraction is slightly more subtle.  Let $$\cA'' := \big\{J\cap H\mid J \in\cA\setminus \{H\}\big\},$$ which is a multiset of hyperplanes in $H$.
Then the contraction $\cL''$ is equal to $\cL(\cA'',\cK\cap H)$,
where the definition of a conditional oriented matroid associated with a set of hyperplanes
is extended to multisets of hyperplanes in the obvious way.  Things become trickier if we begin with a multiset of hyperplanes and contract an element of multiplicity greater than one.
This creates coloops, which is an illustration of why it is necessary to allow coloops when using a recursion involving contractions.
\end{example}



\begin{lemma}\label{minor circuits}
Let $i\in\cI$ be an element of the ground set of $\cL$, and
let $\cC'$ and $\cC''$ be the sets of circuits of $\cL'$ and $\cL''$, respectively.
\begin{enumerate}
  \item We have
  \(
  \cC' =  \{X'\mid \iota(X') \in\cC\} = \{\pi(X)\mid X\in\cC\,\text{ and }\, X_i=0\}.
  \)
  \item
  If $i$ is not a coloop, then $\cC''$ is the set of support-minimal elements of $\{\pi(X)\mid X\in\cG\}$.
  \item If $i$ is not a coloop, $X\in\cC$, and $i\in\underline{X}$, then $\pi(X)\in\cC''$.
\end{enumerate}
\end{lemma}

\begin{proof}
Recall that $\cC$ is defined to be the set of support-minimal elements of $\cG$.
Define the analogous sets $\cG'$ and $\cG''$ for $\cL'$ and $\cL''$, so that
$\cC'$ and $\cC''$ are the support-minimal elements of $\cG'$ and $\cG''$, respectively.
To prove part (1), it will suffice
to show that $$\cG' =  \{X'\mid \iota(X') \in\cG\} = \{\pi(X)\mid X\in\cG\,\text{ and }\, X_i=0\}.$$
First suppose $X\in\cG$ and $X_i=0$.  For any $Y\in\cL$, we have
$(X\circ Y)_i = Y_i$ but $X\circ Y\neq Y$,
which implies that $\pi(X)\circ\pi(Y) = \pi(X\circ Y)\neq \pi(Y)$.
Conversely, suppose that $X'\in\cG'$.  For any $Y\in\cL$, we have
$\pi(\iota(X')\circ Y) = X'\circ \pi(Y)\neq \pi(Y)$, so $\iota(X')\circ Y\neq Y$.  Thus $\iota(X')\in\cG$.

To prove part (2), it will suffice to show that $\cG'' = \{\pi(X)\mid X\in\cG\}$ whenever $i$ is not a coloop.
First suppose that $X''\in \cG''$.  Consider the two signed sets $X,Y$
characterized by the properties that $\pi(X) = X'' = \pi(Y)$, $X_i = +$, and $Y_i=-$.
We claim that at least one of these two signed sets lies in $\cG$.  If not, then there exist $Z,W\in\cL$
such that $X\circ Z = Z$ and $Y\circ W = W$.
Applying the strong elimination axiom to $Z$ and $W$ gives us a covector $U\in\cL$ with $U_i=0$ and $X''\circ \pi(U) = \pi(U)$,
contradicting the hypothesis that $X''\in\cG''$.

Conversely, we need to show that $\pi(X)\in \cG''$ for all $X\in\cG$.
Suppose for the sake of contradiction that we have some $Y''\in\cL''$ such that $\pi(X)\circ Y''=Y''$.
We know that $X\circ\iota(Y'')\neq \iota(Y'')$, but the previous equality tells us that they agree in all but the $i^\text{th}$ coordinate,
so we must have $X_i\neq 0$.
Since $i$ is not a coloop, we may choose a covector $Y\in\cL$ with $Y_i\neq 0$.  Assume first that $Y_i = X_i$,
and let $Z = \iota(Y'')\circ Y\in\cL$.  Then $X\circ Z = Z$, contradicting the fact that $X\in\cG$.  If instead
$Y_i = -X_i$, then we can take $Z = \iota(Y'')\circ -Y\in\cL$, and again $X\circ Z = Z$.

Finally, we prove part (3).  Suppose that $X\in\cC$ and $i\in\underline{X}$.  We need to show that
the element $\pi(X)\in\cG''$ is support-minimal.  Suppose not, and let $Z''\in\cC''$ be a circuit whose support is strictly
contained in that of $\pi(X)$.  We have shown in part (2) that there is some $Z\in\cC$ with $Z'' = \pi(Z)$.
This implies that $\underline{Z}\subsetneq\underline{X}$, contradicting the fact that $X\in\cC$.
\end{proof}

\begin{remark}
Even when $\cL$ is an oriented matroid, there can exist a circuit $X\in\cC$ such that $\pi(X)\notin\cC''$.
For example, let $\cL = \cL(\cA,V)$, where $\cA$ is a multiset consisting a a single hyperplane $H\subset V$ with multiplicity 3.
Here $\cL$ has three circuits (up to sign), each of which is supported on a set of cardinality 2.
On the other hand, $\cL''$ has two circuits (up to sign), each of which is supported on a set of cardinality 1 (a coloop).  One of the three pairs of circuits of $\cL$ projects to a pair of non-minimal elements of $\cG''$.
\end{remark}


\begin{lemma}\label{complement}
If $\pm X\in\cC$ are nonzero circuits, then there exists $Y\in\cL$ such that 
$\underline{X}\cap\underline{Y} = \emptyset$.
\end{lemma}

\begin{proof}
We proceed by induction on the cardinality of the ground set.
The base case is the empty ground set.
There are two conidtional oriented matroids on an empty ground set, both of which satisfy the hypothesis since neither of them has a nonzero circuit.

Let $\cL$ be a conditional oriented matroid on a nonempty ground set, and assume that the lemma holds for every contraction of $\cL$.
Let $\pm X\in\cC$ be nonzero circuits, and choose any element $i\in\underline{X}$.
Either $i$ is a coloop or $i$ is not a coloop, and we treat these cases separately.

When $i$ is a coloop, both $(\{i\},\emptyset)$ and $(\emptyset,\{i\})$ are circuits, and must therefore be equal to $\pm X$.
Since $X$ is a circuit, $(\emptyset,\emptyset)$ is not in $\cG$, so $\cL$ is nonempty.
Any element $Y\in\cL$ has $Y_i = 0$, and therefore satisfies the condition of the lemma.

Now assume that $i$ is not a coloop, and let $\cL''$ denote the contraction of $\cL$ at the element $i$.
By Lemma \ref{minor circuits}(3), we have $\pm\pi(X)\in\cC''$.
From our inductive hypothesis, there exists a covector $Y''\in\cL''$ with $\underline{\pi(X)}\cap\underline{Y''}=\emptyset$.
Then $Y := \iota(Y'')$ is a covector with $\underline{X}\cap\underline{Y} = \emptyset$.
\end{proof}

\begin{lemma}\label{lifts}
If both $\pm X''\in\cC''$ are nonzero circuits,
then there exist $\pm X\in\cC$ with $\pi(X) = X''$.
\end{lemma}

\begin{proof}
If $i$ is a coloop, then we may take $X = \iota(X'')$.  Thus we may assume that $i$ is not a coloop.
By Lemma \ref{minor circuits}(2), there is some $X\in\cC$ with $\pi(X)=X''$.
We need only show that $-X\in\cC$, as well.
By Lemma \ref{complement}, we may choose a covector $Y''\in\cL''$ with $\underline{X''}\cap\underline{Y''}=\emptyset$.
Consider the covector $Y := \iota(Y'')\in\cL$, which has the property that $\underline{X}\cap\underline{Y}=\emptyset$.
Let $\cM$ be the conditional oriented matroid
obtained from $\cL$ by deleting all of the elements of $\underline{Y}$.
The covector $Y\in\cL$ projects to the covector $(\emptyset,\emptyset)\in\cM$, so $\cM$ is an oriented
matroid.  Since we have only deleted elements outside of the support of $X$, Lemma \ref{minor circuits}(1)
tells us that the projection of $X$ is a circuit of $\cM$.
Since the collection of circuits of an oriented matroid is closed under negation, the projection of $-X$ is also a circuit of $\cM$.
Applying Lemma \ref{minor circuits}(1) again tells us that $-X$ is a circuit of $\cL$.
\end{proof}

\subsection{Proof of Proposition \ref{nbc-tope}}
In this section, we prove Proposition \ref{nbc-tope} by showing that the cardinalities of both $\cN$ and $\cT$ satisfy the same deletion/contraction recurrence with the same initial conditions.
We start with the recurrence for $\cT$.
For elements $X\in\tau$ and $i\in\cI$, say that $i$ is a {\bf wall} of $X$ if 
$\iota(\pi(X))\in\cL$.
Define
\begin{eqnarray*}\tau_+ &=& \{X\in\tau \mid \text{$i$ is a wall of $X$ and $X_i = +$}\},\\
\tau_- &=& \{X\in\tau \mid \text{$i$ is a wall of $X$ and $X_i = -$}\},\\
\tau_{=} &=& \{X\in\tau \mid \text{$i$ is not a wall of $X$}\}.
\end{eqnarray*}
Let $\tau'$ and $\tau''$ denote the sets of topes of $\cL'$ and $\cL''$, respectively.

\begin{proposition}\label{deletion-contraction-topes}
If $i\in\cI$ is not a coloop, then $\pi$ restricts to bijections $$\tau_+(\cL)\to \tau''\and \tau_-\cup\tau_=\to\tau'.$$
In particular,
$|\tau| = |\tau'| + |\tau''|.$
\end{proposition}

\begin{proof}
Since $i$ is not a coloop, we can fix a covector $W\in\cL$ with $W_i\neq 0$.  We will treat only the case where $W_i = +$.
If $W_i=-$, the proof can be modified by replacing $W$ with $-W$ every time it appears (even though $-W$ need not be in $\cL$).

We begin with the contraction.
Suppose that $X\in\tau_+$.  By the strong elimination axiom, there is a (unique) covector $Z\in\cL$ with $Z_i=0$
and $Z_j=X_j$ for all $j\neq i$.  Then $\pi(X) = \pi(Z) \in \tau'$.  This shows that $\tau_+(\cL)\to \tau''$ is a well defined injection.
For any $Y''\in\tau''$, we have $\iota(Y'')\circ W\in \tau_+$ and
$\pi(\iota(Y)\circ W) = Y$, thus our map is also surjective.

We now turn to the deletion.
Suppose that $X\neq X'$ and $\pi(X) = \pi(X')$.  This implies that $\Sep(X,X') = \{i\}$,
thus $X$ and $X'$ cannot both be elements of $\tau_-$.
On the other hand, strong elimination implies that $\iota(\pi(X)) = \iota(\pi(X'))\in\cL$, so neither $X$ nor $X'$ lies in $\tau_=$.
Thus our map is injective.
To prove surjectivity, let $Y'\in\tau'$ be given.  By definition, there exists $X\in\cL$ with $\pi(X) = Y$.
If $X\in \tau_-\cup\tau_=$, we are done.  If $X\in\tau_+$, then
$X' := \iota(\pi(X)) \circ -W \in \tau_-(\cL)$ and $\pi(X') = Y$, so we are again done.
Thus we may assume that $X_i = 0$.  In this case, $X\circ -W \in \tau_-(\cL)$ and $\pi(X\circ -W) = Y$.
\end{proof}

We next state a lemma that we will need to prove the recursion for $\cN$.

\begin{lemma}\label{lem:boolean-covectors}
Let $\cJ\subset\cI$, and let $U$ be any signed set on the ground set $\cJ$.
If $\cJ$ does not contain the support of any circuit, then there
exists a covector $Y\in\cL$ with $Y_j=U_j$ for all $j\in\cJ$.
\end{lemma}

\begin{proof}
Let $\cM$ be the conditional oriented matroid on the ground set $\cJ$ obtained from $\cL$ by deleting every element of $\cI\setminus\cJ$.
By Lemma \ref{minor circuits}(1), the circuits of $\cM$ are in bijection with the circuits of $\cL$ whose supports
are contained in $\cJ$, but there are no such circuits.  This implies that every signed set on the ground set $\cJ$ is a covector of $\cM$.
In particular, $U\in\cM$.  By definition of the deletion, there is some $Y\in\cL$ that projects to $U$.
\end{proof}

Now we turn to the recursion for $\cN$, the collection of NBC sets of a conditional oriented matroid $\cL$ with respect to a fixed ordering of the ground set $\cI$.  Let $i\in\cI$ be the maximal element with respect to this ordering, and
let $\cN'$ and $\cN''$ denote the collections of NBC sets for $\cL'$ and $\cL''$, respectively.

\begin{proposition}\label{deletion-contraction-nbc}
If $i$ is not a coloop, then
$$
  \cN' = \{S\in\cN \mid i\notin S\}\and
  \{S''\cup \{i\}\mid S''\in \cN''\}  =  \{S\in\cN \mid i\in S\}.$$
In particular, $|\cN| = |\cN'| + |\cN''|$.
\end{proposition}

\begin{proof}
Since $i$ is the maximal element of $\cI$ and $i$ is not a coloop,
whenever $\pm X$ are circuits with $i\in\underline{X}$,
we also have $i\in\becircled{X}$.
Thus the first equality follows from Lemma \ref{minor circuits}(1).

To prove the second equality, we show containment in both directions.
We start by taking $S''\cup\{i\}\in \cN$ and showing that $S''\in\cN''$.
Suppose for the sake of contradiction that $S''$ contains the support of some $X''\in\cC''$.
By Lemma \ref{minor circuits}(2), there is a circuit $X\in\cC$
with $\pi(X) = X''$.  Then $S''\cup\{i\}$ contains the support of $X$, contradicting the hypothesis that $S''\cup\{i\}\in \cN$.
Next, suppose for the sake of contradiction that $S''$ contains $\becircled{X''}$ for some nonzero $\pm X''\in\cC''$.
Lemma \ref{lifts} tells us that there exist nonzero circuits $\pm X\in\cC$ with $\pi(X) = X''$.
Since $S''$ contains $\becircled{X''}$, the strictly larger set $S''\cup\{i\}$ contains $\becircled{X}$, contradicting the fact that $S''\cup\{i\}\in \cN$.

Conversely, let $S''\in\cN''$ be given.  We must now show that $S''\cup\{i\}\in\cN$.
If $S''\cup\{i\}$ contains the support of some $X\in\cC$,
then $S''$ contains the support of $\pi(X)\in\cG''$, and therefore also the support of some element of $\cC''$.  This contradicts the hypothesis that $S''\in\cN''$.
Finally, suppose for the sake of contradiction that $S''\cup\{i\}$ contains $\becircled{X}$ for some nonzero $\pm X\in\cC$.
If $i\in\underline{X}$, then Lemma \ref{minor circuits}(3) implies that $\pm\pi(X)\in\cC''$.  But $S''$ contains
the support of $\pi(X)$, contradicting the fact that $S''\in\cN''$.
So we may assume that $i\notin\underline{X}$, and therefore
that $\becircled{X}\subset S''$.

We break the remainder of the proof up into two cases, depending on whether or not there exists a covector $Y\in\cL$ such that $Y_i=0$ and $\underline{X}\cap\underline{Y} = \emptyset$.
\begin{itemize}
\item {\sf Case 1.}  Suppose such a covector $Y\in\cL$ exists.
Mimicking the proof of Lemma \ref{lifts}, let $\cM$ be the conditional oriented matroid
obtained from $\cL$ by deleting all of the elements of $\underline{Y}$.  The covector $Y\in\cL$ projects to the covector $(\emptyset,\emptyset)\in\cM$, so $\cM$ is an oriented
matroid.  Since we have only deleted elements outside of the support of $X$, Lemma \ref{minor circuits}(1) tells us that the projection of $X$ is a circuit of $\cM$.
Since the collection of circuits of an oriented matroid is closed under negation, the projection of $-X$ is also a circuit of $\cM$.
Note that we have not deleted the element $i$, so we can consider the contraction $\cM''$ of $\cM$ at $i$, which is again an oriented matroid.
By Lemma \ref{minor circuits}(2), there exist circuits $\pm Z''$ of $\cM''$ whose support is contained in the
support of $X$.  We next observe that $\cM''$ may also be realized as an iterated deletion of $\cL''$,
thus we may use Lemma \ref{minor circuits}(1) to extend $\pm Z''$ by zero and obtain circuits $\pm W''$ of $\cL''$.
We have $\underline{W''}\subset\underline{X}$ and therefore $\becircled{W''}\subset\becircled{X}\subset S''$, contradicting the fact
that $S''\in\cN''$.
\item {\sf Case 2.} Suppose no such $Y\in\cL$ exists.
We will show that there is a circuit $Z''\in\cC''$ with $\underline{Z''}\subset\becircled{X}\subset S''$, contradicting the fact that $S''\in\cN''$.  Suppose for the sake of contradiction that there is no such $Z''$.  By Lemma \ref{lem:boolean-covectors}, there is a covector $Y''\in\cL''$
with $Y_j=0$ for all $j\in\becircled{X}$.
By the definition of the contraction, we have $Y := \iota(Y'')\in\cL$.
Let $m:=\min(X)$.  We have $Y_i=0$ and $\becircled{X}\cap \underline{Y}=\emptyset$, but we cannot have $\underline{X}\cap\underline{Y}=\emptyset$,
so we must have $m\in\underline{Y}$.

Suppose $Y_m = X_m$.  By another application of Lemma \ref{lem:boolean-covectors}, there is a covector $U''\in\cL''$ with $U''_j = X_j$ for all $j\in\becircled{X}$.
Let $U := \iota(U)\in\cL$.
Then $Y\circ U\in\cL$ and $X\circ (Y\circ U) = Y\circ U$, contradicting the fact that $X\in\cC$.  Finally, suppose that $Y_m = -X_m$.
This time, we use Lemma \ref{lem:boolean-covectors} to produce a covector $U''\in\cL''$ with $U''_j = -X_j$ for all $j\in\becircled{X}$.
Let $U := \iota(U'')\in\cL$.
Then $-X \circ (Y\circ U) = Y\circ U$, contradicting the fact that $-X\in\cC$.
\end{itemize}
This completes the proof.
\end{proof}

\begin{proof}[Proof of Proposition \ref{nbc-tope}.]
We proceed by induction on the cardinality of $\cI$.  If $\cI$ is empty, there are exactly two conditional oriented matroids
on $\cI$.  For one of them, the zero signed set $X = (\emptyset,\emptyset)$ is a covector and not a circuit, in which case $X$ is the
unique tope and $\emptyset$ is the unique NBC set.  For the other one, $X$ is a circuit and not a covector, and
both $\cT$ and $\cN$ are empty.

Now suppose that $\cI$ is nonempty and $i$ is the maximal element.  If $i$ is a coloop, then $\cN$ and $\cT$ are both
empty by Remark \ref{zero case}.
If $i$ is not a coloop, then the proposition follows from the inductive hypothesis using Propositions \ref{deletion-contraction-topes}
and \ref{deletion-contraction-nbc}.
\end{proof}

\section{Proof of Theorem \ref{thm:presentations}}\label{sec:presentation}
The goal of this section is to prove Theorem \ref{thm:presentations}.
It suffices to give the presentation for $\Rees\GR(\cL)$, as the rest of Theorem \ref{thm:presentations} will follow from specializing $u$ to 0 or 1.
We regard $\Rees\GR(\cL)$ as the subring of the ring of functions $\tau \to \Z[u]$, generated by $u$ times the Heaviside functions $h^\pm$.
We will be concerned with the surjective $\Z[u]$-algebra homomorphism $\rho:R\to \Rees \GR(\cL)$ sending $e_i^\pm$ to $uh_i^\pm$.

\begin{lemma}\label{kernel}
The ideal $I_\cL + J_\cL$ is contained in the kernel of $\rho$.
\end{lemma}

\begin{proof}
Suppose that $X\in\cC$ and $Y\in\tau$.  We have
$$\rho(e_X)(Y) = (-1)^{|X^-|}u^{|\underline{X}|}\prod_{i\in X^+}h_i^+(Y)\prod_{i\in X^-} h_i^-(Y),$$
which is nonzero if and only if $Y_i = +$ for all $i\in X^+$ and $Y_i = -$ for all $i\in X^-$.
If this were the case, we would have $X \circ Y = Y$, which contradicts the hypothesis that $X\in\cC$.
This proves that $I_\cL$ is contained in the kernel of $\rho$.

Now suppose that $\pm X\in \cC$ are nonzero circuits.  Then
$$u\rho(f_X) = \rho(uf_X) = \rho(e_X - e_{-X}) = \rho(e_X) - \rho(e_{-X}) = 0.$$
Since $\Rees \GR(\cL)$ is a torsion-free $Z[u]$-algebra, this implies that $\rho(f_X) = 0$.  Thus $J_\cL$ is contained in the
kernel of $\rho$.
\end{proof}

\begin{remark}\label{proof using empty intersections}
When $\cL = \cL(\cA,\cK)$ as in Example \ref{basic ex}, it is also possible to prove Lemma \ref{kernel}
by using Theorem \ref{cohomology} to interpret $\Rees \GR(\cL(\cA,\cK)) \cong \Rees\VG(\cA,\cK)$ as
the equivariant cohomology ring $H_T^*(M_3(\cA,\cK); \Z)$.
From this perspective, our homomorphism takes $e_H^\pm$ to the class
\[
[\,\pm g_H^{-1}\R_{>0}\,]_T\in H^2_T(M_3(\cA,\cK); \Z)\,.
\]
The fact that $\rho(e_X)=0$ for any vector $X$ follows from Lemma \ref{empty intersection}.
\end{remark}

Lemma \ref{kernel} implies that $\rho$ descends to a surjective $\Z[u]$-algebra homomorphism
$$\bar\rho:R\Big{/}I_{\cL}+J_{\cL}\to \Rees \GR(\cL).$$
Now we prove that $\bar\rho$ is also injective.

Choose a linear ordering $<$ on $\cI$ as in Section \ref{sec:COMs}, along with a degree monomial order $\prec$ on $\Z[e^+_i]_{i\in\cI}\cong R_1$ such that $e_i^+ \prec e_{j}^+$ if and only if $i<j$.
For any polynomial $f\in R_1$, we will write $\init(f)$ to denote its leading term.
Recall that we defined elements $e_X, f_X\in R$; we now use the same notation to denote the images of these elements in $R_1$.
Then
$$\init(e_X) = \prod_{i\in\underline{X}} e_i^+\and \init(f_X) = \pm \prod_{i\in \becircled{X}} e_i^+\,,$$
where we have a minus sign in $\init(f_X)$ if and only if $\min(X) \in X^-$.  This implies that the NBC monomials
$$\left\{\;\prod_{i\in S} e_i^+\;\Big{|}\; S\in\cN\right\}$$
span $R\Big{/} I_\cL + J_\cL$ as a $\Z[u]$-module.

Before proving Theorem \ref{thm:presentations}, we state and prove one more lemma which is well known to experts, but which we include here for completeness.
Let $A$ be an integral domain with fraction field $K$, and let $P$ be a finitely generated $A$-module.
The {\bf rank} of $P$ is the dimension of $P \otimes_A K$.
We will be interested in the domain $\Z[u]$ and the module $\Rees\GR(\cL)$, whose rank is the cardinality of $\tau$.

\begin{lemma}\label{basic}
If $P$ is a free $A$-module of rank $r$ and $Q$ is an arbitrary $A$-module of rank $r$,
then any surjection $P\to Q$ is an isomorphism.
\end{lemma}

\begin{proof}
Let $N$ be the kernel.  The field $K$ is a flat $A$-module, so we obtain a short exact sequence
$$0\to N\otimes_A K \to P\otimes_A K \to Q\otimes_A K\to 0.$$
The second map is a surjection of vector spaces of dimension $r$, therefore an isomorphism, so $N\otimes_A K = 0$.
Since $N$ is a submodule of a free module, it is torsion-free, thus $N = 0$.
\end{proof}

\begin{proof}[Proof of Theorem \ref{thm:presentations}]
As observed above, it is sufficient to show that the ring homomorphism $\bar\rho$
is in fact an isomorphism.
Let $r$ be the cardinality of $\cN$, which is also equal to the cardinality of $\tau$ by Proposition \ref{nbc-tope}.
Then we have $\Z[u]$-module surjections
$$\Z[u]^r\to R\Big{/} I_{\cL}+J_{\cL}\overset{\bar\rho}{\longrightarrow} \Rees\GR(\cL),$$
where the first map takes the $r$ basis vectors to the $r$ NBC monomials.
Lemma \ref{basic} says that the composition is an isomorphism, and therefore so is $\bar\rho$.
\end{proof}

\newcommand{\etalchar}[1]{$^{#1}$}
\def\cprime{$'$}
\providecommand{\bysame}{\leavevmode\hbox to3em{\hrulefill}\thinspace}
\providecommand{\MR}{\relax\ifhmode\unskip\space\fi MR }
\providecommand{\MRhref}[2]{%
  \href{http://www.ams.org/mathscinet-getitem?mr=#1}{#2}
}
\providecommand{\href}[2]{#2}

\end{document}